\documentclass[12pt,reqno]{amsart} 
\usepackage{amssymb,amscd,url}

\makeatletter
\@namedef{subjclassname@2010}{%
  \textup{2010} Mathematics Subject Classification}
\makeatother

\newtheorem{theorem}{Theorem}
\newtheorem{lemma}[theorem]{Lemma}

\newtheorem{proposition}[theorem]{Proposition}
\newtheorem{corollary}[theorem]{Corollary}

\theoremstyle{definition}
\newtheorem*{definition}{Definition}
\newtheorem{example}[theorem]{Example}

\theoremstyle{remark}
\newtheorem{remark}[theorem]{Remark}

\numberwithin{equation}{section}

  {\end{list}}

\newenvironment{parts}[0]{%
  \begin{list}{}%
    {\setlength{\itemindent}{0pt}
     \setlength{\labelwidth}{1.5\parindent}
     \setlength{\labelsep}{.5\parindent}
     \setlength{\leftmargin}{2\parindent}
     \setlength{\itemsep}{0pt}
     }%
   }%
  {\end{list}}
\newcommand{\Part}[1]{\item[\upshape#1]}

\newcommand{\e}{\epsilon}

\newcommand{\D}{\Delta}

\newcommand{\FF}{\mathbb{F}}
\newcommand{\GG}{\mathbb{G}}

\newcommand{\QQ}{\mathbb{Q}}

\newcommand{\ZZ}{\mathbb{Z}}

\newcommand{\End}{\operatorname{End}}

\newcommand{\LCM}{\operatorname{LCM}}
\newcommand{\Lift}{\operatorname{Lift}}

\newcommand{\LS}[2]{{\genfrac{(}{)}{}{}{#1}{#2}}} 
\newcommand{\tLS}[2]{(#1{}|{}#2)} 

\newcommand{\notdivide}{\nmid}

\newcommand{\ord}{\operatorname{ord}}

\newcommand{\Spec}{\operatorname{Spec}}
\newcommand{\tors}{{\textup{tors}}}

\begin{document}


\baselineskip=17pt


\title[Elliptic Carmichael Numbers]
{Elliptic Carmichael Numbers and Elliptic Korselt Criteria}
\date{}

\author[J. H. Silverman]{Joseph H. Silverman}
\address{Mathematics Department, Box 1917\\
         Brown University\\  Providence, RI 02912 USA}
\email{jhs@math.brown.edu}


\begin{abstract}
Let~$E/\QQ$ be an elliptic curve, let $L(E,s)=\sum a_nn^{-s}$ be the
$L$-series of~$E/\QQ$, and let $P\in E(\QQ)$ be a point.  An
integer~$n>2$ having at least two distinct prime factors will be be
called an \emph{elliptic pseudoprime for $(E,P)$} if~$E$ has good
reduction at all primes dividing~$n$ and $(n+1-a_n)P\equiv
0\pmod{n}$. Then~$n$ is an \emph{elliptic Carmichael number for~$E$}
if~$n$ is an elliptic pseudoprime for every $P\in E(\ZZ/n\ZZ)$. In
this note we describe two elliptic analogues of Korselt's criterion
for Carmichael numbers, and we analyze elliptic Carmichael numbers of
the form~$pq$.
\end{abstract}


\maketitle


\section{Introduction}
\label{section:introduction}

Classically, a composite integer~$n>2$ is called a \emph{pseudoprime
  to the base~$b$} if
\[
  b^{n-1}\equiv 1 \pmod{n}.
\]
A \emph{Carmichael number} is an integer~$n$ that is a pseudoprime to
all bases that are relatively prime to~$n$. Explicit examples of
Carmichael numbers were given by Carmichael~\cite{MR1517641} in~1912,
although the concept had been studied earlier by
Korselt~\cite{korselt1899} in~1899. In particular, Korselt gave the
following elementary criterion for Carmichael numbers, which was
rediscovered by Carmichael.

\begin{proposition}[Korselt's Criterion]
\label{proposition:korselt}
A positive composite number~$n$ is a Carmichael number if and only
if~$n$ is odd, square-free, and every prime~$p$ dividing~$n$ has the
property that $p-1$ divides $n-1$.
\end{proposition}

In~1994, Alford, Granville, and Pomerance~\cite{MR1283874} proved the
long-stand\-ing conjecture that there are infinitely many Carmichael
numbers.

The definitions of pseudoprimes and Carmichael numbers are related to
the orders of numbers in the multiplicative group~$(\ZZ/n\ZZ)^*$.  It
is thus natural to extend these constructions to the setting of other
algebraic groups, for example to elliptic
curves. Gordan~\cite{MR946604} appears to have been the first to
define elliptic pseudoprimes, at least in the setting of elliptic
curves having complex multiplication. See
Remark~\ref{remark:gordanellpsdprm} for a description of Gordan's
definition, which includes a supersingularity condition, and 
for additional references.

In this note we define elliptic pseudoprimes
(Section~\ref{section:ellpseudoprimes}) and elliptic Carmichael
numbers (Section~\ref{section:ellcarmichaelnumbs}) on arbitrary
elliptic curves~$E/\QQ$.  Our definition (mostly) reduces to Gordan's
definition in the CM setting. We give two Korselt-type criteria for
elliptic Carmichael numbers. The first, in
Section~\ref{section:ellkorseltnumbs}, only goes one direction
(Korselt implies Carmichael), but is relatively easy to check in
practice. The second version, described in
Section~\ref{section:ellkorseltnumbs2}, is bi-directional, but less
practical.  In Section~\ref{section:pq} we discuss elliptic Carmichael
numbers~$pq$ that are the product of exactly two primes. (It is an
easy exercise to show that there are no classical Carmichael numbers of
the form~$pq$.) Finally, we give some numerical examples of elliptic
Carmichael numbers in Section~\ref{section:numericalexamples}.

Without going into details (which are given later), we note that our
construction replaces the quantity~$n-1$ in the classical pseudoprime
definition $b^{n-1}\equiv1\pmod{n}$ with the quantity $n+1-a_n$ in the
case of elliptic curves, where~$a_n$ is the usual coefficient of the
$L$-series of~$E/\QQ$. An integer~$n$ is then an elliptic
pseudoprime for the curve~$E$ and point~$P\in E(\ZZ/n\ZZ)$ if 
$E$ has good reduction at all primes dividing~$n$ and
\begin{equation}
  \label{eqn:n1anPeq0modn}
  (n+1-a_n)P\equiv 0\pmod{n},
\end{equation}
where the congruence~\eqref{eqn:n1anPeq0modn} takes place
in~$E(\ZZ/n\ZZ)$.  Notice that if we take~$n$ to be a prime~$p$,
then~\eqref{eqn:n1anPeq0modn} is automatically true, because
$\#E(\ZZ/p\ZZ)=p+1-a_p$. Thus the analogy between the multiplicative
group and elliptic curves that we are using may be summarized by
noting that
\begin{equation}
  \label{eqn:EMpp1Epp1ap}
  \#\GG_m(\ZZ/p\ZZ)=p-1
  \qquad\text{and}\qquad
  \#E(\ZZ/p\ZZ)=p+1-a_p,
\end{equation}
replacing~$p$ by~$n$ (and removing the equality signs), and asking if
the resulting quantity~$n-1$, respectively~$n+1-a_n$, is still an
annihilator of $\GG_m(\ZZ/n\ZZ)$, respectively~$E(\ZZ/n\ZZ)$.

\begin{remark}
In this paper, when we write~$E(\ZZ/n\ZZ)$, we will always 
assume that~$E$ has good reduction at all primes dividing~$n$. It
follows that a minimal Weierstrass equation for~$E/\QQ$ defines a
group scheme
\[
  E \longrightarrow \Spec(\ZZ/n\ZZ), 
\]
so it makes sense to talk about the group of sections, which is what
we mean by the notation~$E(\ZZ/n\ZZ)$. Further, if~$n$ factors as $n=
p_1^{e_1}\cdots p_t^{e_t} $ with~$p_1,\ldots,p_t$ distinct primes,
then there is a natural isomorphism (essentially the Chinese remainder
theorem) 
\[
  E(\ZZ/n\ZZ) \cong E(\ZZ/p_1^{e_1}\ZZ) \times \cdots \times
    E(\ZZ/p_t^{e_t}\ZZ).
\]
\end{remark}

\section{Elliptic Pseudoprimes}
\label{section:ellpseudoprimes}

In this section we define elliptic pseudoprimes in general and relate
our definition to Gordan's definition of elliptic pseudoprimes on CM
elliptic curves.

\begin{definition}
Let $n\in\ZZ$, let~$E/\QQ$ be an elliptic curve given by a minimal
Weierstrass equation, and let~$P\in E(\ZZ/n\ZZ)$.  Write the
$L$-series of~$E/\QQ$ as $L(E/\QQ,s)=\sum a_n/n^s$.  We say that $n$
is an \emph{elliptic pseudoprime} for~$(E,P)$ if~$n$ has at least two
distinct prime factors and the following two conditions hold:
\begin{align}
\bullet\enspace
    &\text{$E$ has good reduction at every prime~$p$ dividing~$n$.}\notag\\
\bullet\enspace
    &(n+1-a_n)P=0\pmod{n}. \label{eqn:defEPP}
\end{align}
\end{definition}

\begin{remark}
We note that if~$E$ has good reduction at~$p$, then every point
in~$E(\ZZ/p\ZZ)$ is killed by~$p+1-a_p$, since $p+1-a_p=\#E(\ZZ/p\ZZ)$.
\end{remark}

\begin{remark}
\label{remark:gordanellpsdprm}
The first definition of  elliptic pseudoprimes appears to be due to
Gordan~\cite{MR946604}. Gordan's definition, which only applies to
elliptic curves with complex multiplication, is as follows. Let~$E/\QQ$
be an elliptic cruve with complex multiplication by an order in
$\QQ(\sqrt{-D}\,)$, and let $P\in E(\QQ)$ be a non-torsion point.
Then a composite number~$n$ is a Gordan elliptic pseudoprime for
the pair~$(E,P)$ if 
\[
  \LS{-D}{n}=-1 
  \qquad\text{and}\qquad
  (n+1)P\equiv0 \pmod{n}.
\]
Gordan's motivation for this definition was to study elliptic pseudoprimes
as tools for primality and factorization algorithms. Under GRH, he
proves that the set of elliptic pseudoprimes has density~$0$, and
gives an example of a pair~$(E,P)$ having infinitely many elliptic
pseudoprimes.
\par
For simplicity, we consider Gordan's definition for a curve~$E$ that
has CM by the full ring of integers of $\QQ(\sqrt{-D}\,)$. Then for
primes $p\ge5$ of good reduction, we have $a_p(E)=0$ if and only if~$p$
is inert in $\QQ(\sqrt{-D}\,)$, which is equivalent
to~$\tLS{-D}{p}=-1$. Thus the condition~$\tLS{-D}{n}=-1$ implies that at
least one prime~$p$ dividing~$n$ satsifies~$a_p(E)=0$. If we also
assume that~$p^2\nmid n$, then~$a_n=0$, since~$a_n$ is a
multiplicative function.  (More generally, if $a_p=0$, then
$a_{p^{2k+1}}=0$ and $a_{p^{2k}}=(-p)^k$ for all $k\ge0$.)
\par
To recapitulate, we have
\[
  \LS{-D}{n}=-1\quad\text{and}\quad\text{$n$ square-free}
  \quad\Longrightarrow\quad  a_n=0.
\]
Thus for (most) square-free values of~$n$, Gordan's condition
$(n+1)P\equiv0\pmod{n}$ is the same as our condition
$(n+1-a_n)P\equiv0\pmod{n}$, because his Jacobi symbol condition
$\tLS{-D}{n}=-1$ forces $a_n=0$.
\par
For other articles that study Gordan elliptic pseudoprimes and related
quantities, see
\cite{MR1104697,
  MR2496340,
  MR1094951,
  MR1181329,
  MR1464541,
  MR1003565,
  MR1281059,
  MR0970701,
  MR2600561}.
\end{remark}

\section{Elliptic Carmichael Numbers}
\label{section:ellcarmichaelnumbs}

\begin{definition}
Let $n\in\ZZ$ and let~$E/\QQ$ be an elliptic curve. We say
that~$n$ is an \emph{elliptic Carmichael number for~$E$} if~$n$
is an elliptic pseudoprime for~$(E,P)$ for every point~$P\in E(\ZZ/n\ZZ)$.
\end{definition}

Classically, a Carmichael number~$n$ is necessarily odd, since it
satisfies $(-1)^{n-1}\equiv1\pmod{n}$. More intrinsically, this is
true because the multiplicative group~$\GG_m(\QQ)$ has an
element of order~$2$.  The elliptic analog of this fact is the
following elementary proposition.

\begin{proposition}
\label{proposition:effectoftorsion}
Let~$E/\QQ$ be an elliptic curve, and let~$T\in E(\QQ)$ be a torsion 
point of exact order~$m$. If~$n$ is a Carmichael number for~$E$, then
\[
  n \equiv a_n-1 \pmod{m}.
\]
\end{proposition}
\begin{proof}
Suppose that~$n$ is a Carmichael number for~$E$.  
To ease notation, let $N=n+1-a_n$.
By definition,~$n$ has at least two distinct prime factors, say~$p$
and~$q$.  Further, we know that $NT\equiv0\pmod{n}$, and hence
\[
  NT\equiv 0 \pmod{p}
  \quad\text{and}\quad
  NT\equiv 0 \pmod{q}.
\]
\par 
Write $m=p^im'$ with $p\notdivide m'$. Then
$p^iNT\equiv0\pmod{p}$, and also~$p^iNT$ is killed by~$m'$. The
injectivity of prime-to-$p$ torsion under reduction
modulo~$p$~\cite[VII.3.1]{MR2514094} allows us to conclude
that~$p^iNT=0$.
\par
Similarly, writing $m=q^jm''$ with~$q\notdivide m''$, we find
that~$q^jNT=0$. Since~$p$ and~$q$ are distinct, it follows that~$NT=0$.
But by assumption,~$T$ has exact order~$m$,  hence~$m|N$.
\end{proof}

\begin{remark}
An appropriate formulation of
Proposition~\ref{proposition:effectoftorsion} is true more generally
for abelian varieties. Thus let~$A/\QQ$ be an abelian variety, let~$n$
be an integer with at least two distinct prime factors~$p$ and~$q$
such that~$A$ has good reduction at~$p$ and~$q$, and let~$N$ be an
integer that annihilates~$A(\ZZ/n\ZZ)$.  (Here we can take~$A$ to be
the N\'eron model over~$\ZZ$, so~$A$ is a group scheme over~$\Spec\ZZ$
and it makes sense to talk about the group of sections~$A(\ZZ/n\ZZ)$.)
Suppose further that~$A(\QQ)$ has a point of exact order~$m$. Then
$m\mid N$.
\end{remark}

\begin{definition}
Let~$n\in\ZZ$. We will say that~$n$ is a \emph{universal elliptic
  Carmichael number} if~$n$ is an elliptic Carmichael number for every
elliptic curve (elliptic scheme) over~$\ZZ/n\ZZ$.
\end{definition}

\begin{remark}
A natural question is whether there are \emph{any} universal elliptic
Carmichael numbers.  Our guess is that probably none exist, or in any
case, that there are at most finitely many.  This raises the
interesting question of finding nontrivial upper and lower bounds, in
terms of~$n$, for the size of the set
\begin{equation}
  \label{eqn:EnCarm}
  \{ E\bmod n : \text{$n$ is a Carmichael number for $E$} \}.
\end{equation}
For example, suppose that~$n=pq$ is a product of distinct primes.  A
very rough heuristic estimate suggests that the probability that a
given~$E\bmod pq$ has~$pq$ as a Carmichael number is~$O((pq)^{-1})$,
so at least for such~$n$ one might conjecture that the size of the
set~\eqref{eqn:EnCarm} is bounded independently of~$pq$.
\end{remark}

\section{Elliptic Korselt Numbers of Type I}
\label{section:ellkorseltnumbs}

The classical Korselt criterion
(Proposition~\ref{proposition:korselt}) gives an efficient method for
determining if a given integer~$n$ is a Carmichael number, assuming of
course that one is able to factor~$n$ into a product primes. In this
section we give a practical one-way Korselt criterion for elliptic
Carmichael numbers. Any number satisfying this elliptic Korselt
criterion is an elliptic Carmichael number, but the converse need not
be true.

\begin{definition}
Let $n\in\ZZ$, and let~$E/\QQ$ be an elliptic curve. We say that~$n$ is
an \emph{elliptic Korselt number for~$E$ of Type I} if $n$ has at
least two distinct prime factors, and if for every prime~$p$
dividing~$n$, the following conditions hold:
\begin{align}
\bullet\enspace
    &\text{$E$ has good reduction at~$p$.}\notag\\
\bullet\enspace
    &\text{$p+1-a_p$ divides $n+1-a_n$.}
         \label{eqn:defKorselt1} \\
\bullet\enspace
    & \ord_p(a_n-1)\ge\ord_p(n)
       -\begin{cases}
          1&\text{if $a_p\not\equiv1\pmod{p}$,}\\
          0&\text{if $a_p\equiv1\pmod{p}$.}\\
       \end{cases}
      \label{eqn:defKorselt2}
\end{align}
\end{definition}

\begin{remark}
If~$n$ is square-free and~$a_p\not\equiv1\pmod{p}$ for all~$p\mid n$,
then the condition~\eqref{eqn:defKorselt2} is vacuous, since it
reduces to the statement that $\ord_p(a_n-1)\ge0$.
\end{remark}

\begin{remark}
\label{remark:ellcarmnotsqfr}
Classical Carmichael numbers are automatically square-free. The
elliptic analog of this fact is our Korselt
condition~\eqref{eqn:defKorselt2}.  To see the relationship, we extend
the analogy used by Gordan to consider values of~$n$ such that~$E$ is
supersingular at all primes~$p\mid n$. For ease of exposition, we'll
make the slightly stronger assumption that $a_p=0$ for all~$p\mid
n$. (This is only stronger for $p=2$ and $p=3$.)  Then $p\mid a_n$,
since as noted earlier, $a_n$ is a multiplicative function,
and~$a_p=0$ implies that $a_{p^{2k+1}}=0$ and $a_{p^{2k}}=(-p)^k$ for
all $k\ge0$. Hence in this situation we have
\[
  \ord_p(a_n-1)=0\quad\text{and}\quad a_p=0\not\equiv1\pmod{p},
\]
so~\eqref{eqn:defKorselt2} reduces to the statement
that~$\ord_p(n)\le1$. This is true for all $p\mid n$, so~$n$ is
square-free. Of course, this is under the assumption that~$a_p=0$ for
all~$p\mid n$. As we will see later in Example~\ref{example:y2x3x3},
elliptic Carmichael numbers need not in general be square-free.
\end{remark}

\begin{remark}
\label{remark:apeq1modp}
If $p\ge7$, then 
\[
  a_p\equiv1\pmod{p}\quad\Longleftrightarrow\quad
  \text{$E$ is anomalous at~$p$,}
\]
where we recall that~$E$ is anomalous if~$a_p=1$, or equivalently, if
we have $\#E(\ZZ/p\ZZ)=p$.  In particular,
condition~\eqref{eqn:defKorselt2} in the definition of Type~I Korselt
numbers is vacuous if the following three conditions are true for all
prime divisors~$p$ of~$n$:
\[
  \text{$p\ge7$,\quad $E$ is not anomalous at~$p$, \quad$p^2\nmid n$.}
\]
We also observe that the Hasse--Weil estimate $|a_p|\le2\sqrt{p}$
implies that
\[
  \ord_p(p+1-a_p)\le1\quad\text{unless $p=2$ and $a_p=-1$.}
\]
The exceptional case, namely $\ord_2(3-a_2)=2$ when $a_2=-1$, is the
reason that the next proposition deals only with odd values of~$n$.
\end{remark}

\begin{proposition}[Elliptic Korselt Criterion~I]
\label{proposition:korseltimpliescarm}
Let $n\in\ZZ$ be an odd integer, and let~$E/\QQ$ be an elliptic
curve. If~$n$ is an elliptic Korselt number for~$E$ of Type~I,
then~$n$ is an elliptic Carmichael number for~$E$.
\end{proposition}
\begin{proof}
Let~$p$ be a prime of good reduction for~$E$.  Then the group
$E(\ZZ/p\ZZ)$ has order $p+1-a_p$, so the standard filtration on the
formal group of~$E(\QQ_p)$ (see~\cite{MR2514094}) implies that
\begin{equation}
  \label{eqn:pi1p1apPeq0}
  p^{i-1}(p+1-a_p)P\equiv0\pmod{p^i}
  \qquad\text{for all $i\ge1$ and all $P\in E(\QQ_p)$.}
\end{equation}
\par
Now let $P\in E(\ZZ/n\ZZ)$, and write $n=p^in'$ with $i\ge1$ and
$p\notdivide n'$.  Suppose first that $a_p\not\equiv1\pmod{p}$.
Then~$p+1-a_p$ is relatively prime to~$p$, so~\eqref{eqn:defKorselt1}
and~\eqref{eqn:defKorselt2} together imply that
\begin{equation}
  \label{eqn:pi1p1apdivn1an}
  p^{i-1}(p+1-a_p)\quad\text{divides}\quad n+1-a_n.
\end{equation}
\par
Next suppose that $a_p\equiv1\pmod{p}$. As noted earlier, the
Hasse--Weil estimate $|a_p|\le2\sqrt{p}$ then implies that
\begin{equation}
  \label{eqn:ordpp1ap1}
  \ord_p(p+1-a_p)=1.
\end{equation}
(This is where we use the assumption that~$n$ is odd, so~$p\ne2$.)
We compute
\begin{align}
  \label{eqn:ordpn1angepi1}
  \ord_p(n&+1-a_n) \notag\\
  &= \ord_p(p^in'+1-a_n) 
    &&\text{since $n=p^in'$,}\notag\\
  &\ge \min\bigl\{i,\ord_p(a_n-1)\bigr\} 
    &&\text{triangle inequality,} \notag\\
  &\ge \min\bigl\{i,\ord_p(n)\bigr\} 
    &&\text{from Korselt condition \eqref{eqn:defKorselt2},} \notag\\
  &=i
    &&\text{since $n=p^in'$,}\notag\\
  &=\ord_p\bigl(p^{i-1}(p+1-a_p)\bigr)
    &&\text{from \eqref{eqn:ordpp1ap1}.}
\end{align}
\par
Combining~\eqref{eqn:pi1p1apdivn1an} and~\eqref{eqn:ordpn1angepi1},
we have proven that
\[
  p^{i-1}(p+1-a_p) \mid n+1-a_n
  \qquad\text{for all primes $p\mid n$.}
\]
It follows from~\eqref{eqn:pi1p1apPeq0} that
\[
  (n+1-a_n)P\equiv 0 \pmod{p^{\ord_p(n)}}
  \qquad\text{for all primes $p\mid n$.}
\]
Using the Chinese remainder theorem, we conclude
\[
  (n+1-a_n)P\equiv 0 \pmod{n}.
\]
Finally, since~$P\in E(\ZZ/n\ZZ)$ was arbitrary, this completes the proof
that~$n$ is an elliptic Carmichael number for~$E$.
\end{proof}

\section{Elliptic Korselt Numbers of Type II}
\label{section:ellkorseltnumbs2}

The classical Korselt criterion gives both a necessary and sufficient
condition for a number~$n$ to be a Carmichael number. Our
Proposition~\ref{proposition:korseltimpliescarm} gives one
implication, namely Type~I~Korselt implies Carmichael.  The reason
we do not get the converse implication
is because condition~\eqref{eqn:defKorselt1} in the definition of
Type~I Korselt numbers is not, in fact, the exact analog of the
classical condition.  Condition~\eqref{eqn:defKorselt1} comes from the
analogy, already noted in the introduction~\eqref{eqn:EMpp1Epp1ap},
that
\[
  \#\GG_m(\ZZ/p\ZZ)=p-1
  \qquad\text{and}\qquad
  \#E(\ZZ/p\ZZ)=p+1-a_p.
\]
However, the real reason that~$p-1$ appears in the classical Korselt
criterion is because~$p-1$ is the exponent of the
group~$(\ZZ/p\ZZ)^*$, i.e.,~$p-1$ is the smallest positive integer
that annihilates every element of~$(\ZZ/p\ZZ)^*$. This follows, of course,
from the fact that~$(\ZZ/p\ZZ)^*$ is cyclic.
\par
Elliptic curve groups~$E(\ZZ/p\ZZ)$, by way of contrast, need not be
cyclic, although it is true that they are always a product of at most
two cyclic groups. So a more precise elliptic analog of the classical
Korselt criterion is obtained by using the exponent of the
group~$E(\ZZ/p\ZZ)$, rather than its order. This leads to the
following definition and criterion, which while more satisfying in
that it is both necessary and sufficient, is much less practical than
Proposition~\ref{proposition:korseltimpliescarm}.

\begin{definition}
For a group~$G$, we write~$\e(G)$ for the \emph{exponent of~$G$},
i.e.,  the least common multiple of the orders of the elements
of~$G$. Equivalently,~$\e(G)$ is the smallest postive integer such
that~$g^{\e(G)}=1$ for all~$g\in G$.  For an elliptic curve~$E/\QQ$,
integer~$n$, and prime~$p$, to ease notation we will write
\[
  \e_{n,p}(E) = \e\left(E\left(\frac{\ZZ}{p^{\ord_p(n)}\ZZ}\right)\right),  
\]
\end{definition}

\begin{definition}
Let $n\in\ZZ$, and let~$E/\QQ$ be an elliptic curve. We say that~$n$ is
an \emph{elliptic Korselt number for~$E$ of Type II} if $n$ has at
least two distinct prime factors, and if for every prime~$p$
dividing~$n$, the following conditions hold:
\begin{align}
\bullet\enspace
    &\text{$E$ has good reduction at~$p$.}\notag\\
\bullet\enspace
    &\text{$\e_{n,p}(E)$ divides $n+1-a_n$.}
         \label{eqn:defKorselt3} 
\end{align}
\end{definition}

\begin{proposition}[Elliptic Korselt Criterion~II]
\label{proposition:korselt2iffcarm}
Let $n>2$ be an odd integer, and let~$E/\QQ$ be an elliptic
curve. Then~$n$ is an elliptic Carmichael number for~$E$ if and only
if~$n$ is an elliptic Korselt number for~$E$ of Type~II.
\end{proposition}
\begin{proof}
The definitions of both elliptic Carmichael and elliptic Korselt
numbers include the requirement that~$E$ have good reduction at every
prime dividing~$n$, so we assume that this is true without further
comment.
\par
Suppose first that~$n$ is an elliptic Carmichael number.  By
definition, this means that
\begin{equation}
  \label{eqn:n1anP0nforallP}
  (n+1-a_n)P=0\pmod{n}
  \quad\text{for all $P\in E(\ZZ/n\ZZ)$.}
\end{equation}
In other words, the quantity~$n+1-a_n$ annihilates the group
$E(\ZZ/n\ZZ)$. Hence for any prime power~$p^i$ dividing~$n$, the
quantity~$n+1-a_n$ will also annihilate the group $E(\ZZ/p^i\ZZ)$. It
follows that~$n+1-a_n$ is divisible by~$\e_{p,n}(E)$, which is the
exponent of the group~$E(\ZZ/p^i\ZZ)$.  This is true for every prime
dividing~$n$, and hence~$n$ is a Type~II Korselt number for~$E$.
\par
Conversely, suppose that~$n$ is Type~II Korselt. Factoring~$n$
as $n=p_1^{e_1}\cdots p_t^{e_t}$, we have from the Chinese remainder theorem
\[
  E(\ZZ/n\ZZ) = E(\ZZ/p_1^{e_1}\ZZ)\times\cdots\times E(\ZZ/p_t^{e_t}\ZZ),
\]
from which we see that
\begin{equation}
  \label{eqn:eEZnlcm}
  \e\bigl(E(\ZZ/n\ZZ)\bigr) = 
  \LCM\bigl[\e_{n,p_1}(E),\ldots,\e_{n,p_t}(E)\bigr].
\end{equation}
Property~\eqref{eqn:defKorselt3} of Type~II Korselt numbers says
that 
\begin{equation}
  \label{eqn:enpforlcm}
  \e_{n,p}(E)\mid n+1-a_n
  \quad\text{for all $p\mid n$,}
\end{equation}
and combining~\eqref{eqn:eEZnlcm} and~\eqref{eqn:enpforlcm} yields
\[
  \e\bigl(E(\ZZ/n\ZZ)\bigr) \mathbin{\bigm|} n+1-a_n.
\]
It follows that $n+1-a_n$ annihilates $E(\ZZ/n\ZZ)$, which means that~$n$
is an elliptic Carmichael number. 
\end{proof}

\begin{corollary}
If~$n$ is an odd
elliptic Korselt number for~$E/\QQ$ of Type~I, then it is also an
elliptic Korselt number for~$E/\QQ$ of Type~II. 
\end{corollary}
\begin{proof}
Propositions~\ref{proposition:korseltimpliescarm}
and~\ref{proposition:korselt2iffcarm} give the implications
\[\begin{CD}
  \text{Korselt Type I}
  @>\text{Prop.\ \ref{proposition:korseltimpliescarm}}>>
  \text{Carmichael}
  @>\text{Prop.\ \ref{proposition:korselt2iffcarm}}>>
  \text{Korselt Type II}.
  \end{CD}
\]
\end{proof}

In order to understand the definition of elliptic Korselt numbers of
Type~II, we gather some information about the exponents~$\e_{n,p}(E)$.
We begin with a slightly technical definition.

\begin{definition}
Let $p\ge3$ be a prime, and let~$E/\QQ$ be an elliptic curve with
good anomalous reduction at~$p$, i.e., $a_p(E)\equiv1\pmod{p}$. 
(If $p\ge7$, this is equivalent to $a_p(E)=1$.) For each power~$p^i$
with~$i\ge2$, we say that $E$ is \emph{$p^i$-canonical} if
\[
  E(\ZZ/p^i\ZZ)[p] \cong \ZZ/p\ZZ \times \ZZ/p\ZZ,
\]
and  $E$ is \emph{$p^i$-noncanonical} if
\[
  E(\ZZ/p^i\ZZ)[p] \cong \ZZ/p\ZZ.
\]
\end{definition}

\begin{remark}
\label{remark:exactseq}
For primes~$p\ge3$, the formal group of~$E/\QQ_p$ satisfies
$\hat{E}(p\ZZ_p)\cong p\ZZ_p^+$ (see
\cite[Theorem~IV.6.4]{MR2514094}), so there is an exact sequence
\[
  0 \longrightarrow p\ZZ_p^+
    \longrightarrow E(\ZZ_p) \longrightarrow E(\ZZ/p\ZZ)
    \longrightarrow 0.
\]
Reducing modulo~$p^i$ gives
\begin{equation}
  \label{eqn:lesforE1}
  0 \longrightarrow {p\ZZ}/{p^i\ZZ}
    \longrightarrow E(\ZZ/p^i\ZZ) \longrightarrow E(\ZZ/p\ZZ)
    \longrightarrow 0.
\end{equation}
Assume now that~$i\ge2$ and $a_p\equiv1\pmod{p}$,
so in particular 
\[
  \#E(\ZZ/p\ZZ)=p+1-a_p\equiv0\pmod{p}.
\]
Note that the Hasse--Weil estimate says that $p^2\nmid\#E(\ZZ/p\ZZ)$,
so taking the $p$-torsion of~\eqref{eqn:lesforE1} gives
\begin{equation}
  \label{eqn:lesforE2}
  0 \longrightarrow \ZZ/p\ZZ \longrightarrow E(\ZZ/p^i\ZZ)[p]
    \longrightarrow \ZZ/p\ZZ  \longrightarrow 0.
\end{equation}
This shows that $E(\ZZ/p^i\ZZ)[p]\cong(\ZZ/p\ZZ)^k$ with $k=1$ or~$2$,
and hence that~$E$ is either $p^i$-canonical or $p^i$-noncanonical,
i.e., there is no third option. 
\end{remark}

\begin{remark}
\label{remark:canlift} 
For an ordinary elliptic curve~$\tilde C/\FF_p$, the \emph{canonical
  lift}, also sometimes called the \emph{Deuring lift}, is an elliptic
curve~$C/\QQ_p$ whose reduction is~$\tilde C$ and having the property
that $\End(C)\cong\End(\tilde C)$. Equivalently, the Frobenius map
on~$\tilde C$ lifts to an endomorphism of~$C$. Necessarily, the
curve~$C$ has~CM. We denote the canonical lift by $\Lift(\tilde C/\FF_p)$.
Now let~$E/\QQ$ be an elliptic curve. 
A result of Gross~\cite[page~514]{MR1074305} 
implies that the sequence~\eqref{eqn:lesforE2} splits if and only
\[
  j(E) \equiv j\bigl(\Lift(\tilde E/\FF_p)\bigr) \pmod{p^2},
\]
i.e., if and only if~$E\bmod{p^2}$ is isomorphic, modulo~$p^2$, to the
canonical lift of~$E\bmod{p}$.  Thus at least for~$i=2$, the curve~$E$
is $p^2$-canonical according to our definition if $E\bmod{p^2}$ is a
canonical lift in the usual sense.  For further information about
canonical lifts, see for example~\cite{MR1074305,MR1801221}. 
\end{remark}

\begin{lemma}
\label{lemma:enpE}
Let~$p\ge3$ be a prime, and factor
\[
  \e_{n,p}(E) = p^fA\quad\text{with $\gcd(A,p)=1$.}
\]
\vspace{-10pt}
\begin{parts}
\Part{(a)}
If $a_p\not\equiv1\pmod{p}$, then
\[
  A \mid p+1-a_p \quad\text{and}\quad f = \ord_p(n)-1.
\]
\Part{(b)}
If $a_p\equiv1\pmod{p}$, then $A=1$ or~$2$, and
\[
  f = \begin{cases}
    \ord_p(n)-1&\text{if $E$ is $p^{\ord_p(n)}$-canonical,}\\
    \ord_p(n)&\text{if $E$ is $p^{\ord_p(n)}$-noncanonical.}\\
  \end{cases}
\]
\end{parts}
\end{lemma}

\begin{proof}
To ease notation, let~$i=\ord_p(n)$. We use the exact sequence
\begin{equation}
  \label{eqn:lesforE}
  0 \longrightarrow \frac{p\ZZ}{p^i\ZZ}
    \longrightarrow E(\ZZ/p^i\ZZ) \longrightarrow E(\ZZ/p\ZZ)
    \longrightarrow 0
\end{equation}
as described in Remark~\ref{remark:exactseq}.
\par
Suppose first that $a_p\not\equiv1\pmod{p}$. It follows that
\[
  \#E(\ZZ/p\ZZ)=p+1-a_p\not\equiv0\pmod{p},
\]
so the exponent of~$E(\ZZ/p^i\ZZ)$ has the form~$p^{i-1}A$ for some~$A$
dividing~$p+1-a_p$. This completes the proof of~(a).
\par
We now suppose that~$a_p\equiv1\pmod{p}$, so $\#E(\ZZ/p\ZZ)=Ap$.
The Hasse--Weil estimate gives
\[
  A = \frac{p+1-a_p}{p}
  \le \frac{p+1+2\sqrt{p}}{p}
  = \left(1+\frac{1}{\sqrt{p}}\right)^2.
\]
Since~$p\ge3$, we see that~$A\le 2$, so $p\nmid A$; and if $p\ge7$,
then~$A$ must equal~$1$. In any case, we have $A\mid p+1-a_p$.
\par
It follows from the exact sequence~\eqref{eqn:lesforE} that the
exponent of~$E(\ZZ/p^i\ZZ)$ is given by
\[
  \e\bigl(E(\ZZ/p^i\ZZ)\bigr)
  = \begin{cases}
    Ap^i&\text{if the sequence~\eqref{eqn:lesforE} does not split,}\\
    Ap^{i-1}&\text{if the sequence~\eqref{eqn:lesforE} does split.}\\
  \end{cases}
\]
Further, since~\eqref{eqn:lesforE} is (essentially) the extension of
a cyclic group of order~$p^{i-1}$ by a cyclic group of order~$p$, we
see that it splits if and only if~$E(\ZZ/p^i\ZZ)$ has a $p$-torsion
point that does not map to~$0$ in~$E(\ZZ/p\ZZ)$. In other words,
\begin{align*}
  \text{the sequence~\eqref{eqn:lesforE} splits}
  &\quad\Longleftrightarrow\quad
  E(\ZZ/p^i\ZZ)[p]\cong \ZZ/p\ZZ\times\ZZ/p\ZZ\\
  &\quad\Longleftrightarrow\quad
  \text{$E$ is $p^i$-canonical.}
\end{align*}
This observation completes the proof of~(b).
\end{proof}

\section{Elliptic Korselt Numbers of the Form $pq$}
\label{section:pq}
It is an easy consequence of the Korselt criterion that a classical
Carmichael number must be a product of at least three (distinct odd)
primes. This is not true for elliptic Korselt numbers, as seen in the
examples in Section~\ref{section:numericalexamples}. However, elliptic
Korselt numbers of the form~$n=pq$ do satisfy some restrictions, as
in the following result.

\begin{proposition}
\label{proposition:korseltpq}
Let~$E/\QQ$ be an elliptic curve, and let $n=pq$ be a Type~I elliptic
Korselt number for~$E$ that is a product of two distinct primes, say
with $p<q$. Then one of the following is true:
\begin{parts}
\Part{(i)}
$p\le 17$.
\Part{(ii)}
$a_p=a_q=1$, i.e., both $p$ and $q$ are anomalous primes for~$E$.
\Part{(iii)}
$p\ge \sqrt{q}$.
\end{parts}
\end{proposition}
\begin{proof}
We assume that $p>17$ and that at least one of~$a_p$ and~$a_q$ is not
equal to~$1$, and we will prove that~$p$ satisfies the estimate
in~(iii).  We have
\[
  n+1-a_n 
  = pq+1-a_pa_q
  = p(q+1-a_q) + pa_q - p - a_pa_q + 1.
\]
Thus the Korselt condition $q+1-a_q\mid n+1-a_n$ implies that
\begin{equation}
  \label{eqn:q1aqdivpaqpapaq1}
  q+1-a_q \mid pa_q- p - a_pa_q + 1.
\end{equation}
We consider two cases. 
\par
First, suppose that $pa_q- p - a_pa_q + 1=0$. A little bit of algebra
yields
\[
  (p-a_p)(a_q-1)=a_p-1.
\]
We have $p\ne a_p$, since $p\ge5$ by assumption, so $a_p=1$ if and only
if~$a_q=1$. We're also assuming that they are not both equal to~$1$, so
neither is equal to~$1$ and we can solve for~$p$,
\[
  p = a_p + \frac{a_p-1}{a_q-1}.
\]
But then
\[
  p \le |a_p| + \left|\frac{a_p-1}{a_q-1}\right|
  \le |a_p| + |a_p-1|
  \le 2|a_p|+1
  \le 4\sqrt{p}+1.
\]
This contradicts~$p>17$, so we conclude that
$pa_q- p - a_pa_q + 1\ne0$.
\par
It then follows from the Korselt divisibility
condition~\eqref{eqn:q1aqdivpaqpapaq1} that
\[
  |q+1-a_q| \le |pa_q- p - a_pa_q + 1|.
\]
Using the Hasse--Weil estimate for~$a_p$ and~$a_q$, this gives
\[
  q+1-2\sqrt{q}
  \le p\sqrt{q}+\sqrt{pq}+(p-1).
\]
Treating this as a quadratic inequality for~$\sqrt{p}$, we find that
\begin{equation}
  \label{eqn:sqrtpge}
  \sqrt{p} \ge \frac{\sqrt{4q^{3/2}-3q+8}-\sqrt{q}}{\sqrt{q}+1}.
\end{equation}
Asymptotically this gives $\sqrt{p}\ge2\sqrt[4]{q}$, and a little bit
of calculus shows that the right-hand side of~\eqref{eqn:sqrtpge} 
is larger than~$\sqrt[4]{q}$ for all~$q\ge13$. Squaring, we find that
\[
  p \ge \sqrt{q}
  \qquad\text{for all $q\ge13$.}
\]
Since we are assuming that $q>p>17$, this proves property~(iii), which
completes the proof of Proposition~\ref{proposition:korseltpq}.
\end{proof}

\section{Numerical Examples}
\label{section:numericalexamples}

In this section we present several numerical examples of elliptic
Carmichael and elliptic Korselt numbers. These examples were computed using
PARI-GP~\cite{PARI}.

\begin{example}
\label{example:y2x3x3}
Let~$E$ be the elliptic curve
\[
  E:y^2 = x^3 + x + 3.
\]
Its discriminant is $\D_E=-2^4\cdot13\cdot19$. The curve~$E$ has six
Korselt (and hence Carmichael) numbers smaller than~1000.  They are
described in Table~\ref{table:korselt00013}. In particular, note that
the table contains elliptic Korselt (hence Carmichal) numbers
$245=5\cdot7^2$ and $875=5^3\cdot7$ that are not square-free;
cf.\ Remark~\ref{remark:ellcarmnotsqfr}.
\end{example}

\begin{table}
\[
\begin{array}{|c|c|c|c|}  \hline
  n & n+1-a_n & p & p+1-a_p \\ \hline
  15=3\cdot 5 & 16=2^{4}
     & 3 & 4=2^{2}  \\
     & & 5 & 4=2^{2}  \\ \hline
  77=7\cdot 11 & 90=2\cdot 3^{2}\cdot 5
     & 7 & 6=2\cdot 3  \\
     & & 11 & 18=2\cdot 3^{2}  \\ \hline
  203=7\cdot 29 & 216=2^{3}\cdot 3^{3}
     & 7 & 6=2\cdot 3  \\
     & & 29 & 36=2^{2}\cdot 3^{2}  \\ \hline
  245=5\cdot 7^{2} & 252=2^{2}\cdot 3^{2}\cdot 7
     & 5 & 4=2^{2}  \\
     & & 7 & 6=2\cdot 3  \\ \hline
  725=5^{2}\cdot 29 & 720=2^{4}\cdot 3^{2}\cdot 5
     & 5 & 4=2^{2}  \\
     & & 29 & 36=2^{2}\cdot 3^{2}  \\ \hline
  875=5^{3}\cdot 7 & 900=2^{2}\cdot 3^{2}\cdot 5^{2}
     & 5 & 4=2^{2}  \\
     & & 7 & 6=2\cdot 3  \\ \hline
\end{array}
\]
\caption{Type~I Elliptic Korselt numbers for $E:y^2=x^3+x+3$}
\label{table:korselt00013}
\end{table}

\begin{example}
Let~$E$ be the elliptic curve
\begin{equation}
  \label{eqn:Ey2x37x3}
  E:y^2 = x^3 + 7x + 3.
\end{equation}
It has discriminant $\D_E=-25840=-2^{4}\cdot 5\cdot 17\cdot 19$ and
conductor $N=25840$.  It is curve~25840w in Cremona's tables, which
also tell us that its rank is exactly~1. This curve~$E$ has no Type~I
Korselt numbers smaller than~$25000$. We do not know why this is true,
since the curves $y^2=x^3+ax+b$ with
$(a,b)\in\{(6,3),(8,3),(7,2),(7,4)\}$ have lots of Type~I Korselt
numbers smaller than~$10000$. The first few Type~I Korselt
numbers for the curve~\eqref{eqn:Ey2x37x3} are
\[
  \{27563,\; 29711,\; 30233,\; 41683,\; 
  43511,\; 62413,\; 68783,\; 80519,\; 95207\}.
\]
We also mention that this curve has $E(\QQ)_\tors=0$.
\end{example}

\begin{example}
Let~$E$ be the elliptic curve
\[
  E:y^2+xy+3y=x^3+2x^2+4x.
\]
Then there are exactly six numbers $n\le5000$ that are Type~I elliptic
Korselt numbers for~$E$, as described in
Table~\ref{table:korselt12340}. Extending the search, there are~20
Type~I elliptic Korselt numbers for~$E$ that are smaller
than~$100000$,
\begin{multline*}
  \{65, 143, 533, 1991, 4179, 4921, 5251, 5611, 7429, 15839, 22939,
  32339, \\35165, 35303, 41495, 48719, 56959, 69475, 83839, 98879\}.
\end{multline*}
Extending the search up to~$200000$ yields three more examples,
\[
  \{105083,161551, 166493\}.
\]
The non-square-free numbers in this list are
\[
   69475 = 5^2 \cdot 7 \cdot 397,\qquad
   83839 = 7^2 \cdot 29 \cdot 59,\qquad
  161551 = 13 \cdot 17^2 \cdot 43.
\]
\end{example}

\begin{table}
\[
\begin{array}{|c|c|c|c|}  \hline
  n & n+1-a_n & p & p+1-a_p \\ \hline
  65=5\cdot 13 & 54=2\cdot 3^{3}
     & 5 & 9=3^{2}  \\
     & & 13 & 18=2\cdot 3^{2}  \\ \hline
  143=11\cdot 13 & 144=2^{4}\cdot 3^{2}
     & 11 & 12=2^{2}\cdot 3  \\
     & & 13 & 18=2\cdot 3^{2}  \\ \hline
  533=13\cdot 41 & 486=2\cdot 3^{5}
     & 13 & 18=2\cdot 3^{2}  \\
     & & 41 & 54=2\cdot 3^{3}  \\ \hline
  1991=11\cdot 181 & 1992=2^{3}\cdot 3\cdot 83
     & 11 & 12=2^{2}\cdot 3  \\
     & & 181 & 166=2\cdot 83  \\ \hline
  4179=3\cdot 7\cdot 199 & 4180=2^{2}\cdot 5\cdot 11\cdot 19
     & 3 & 4=2^{2}  \\
     & & 7 & 10=2\cdot 5  \\
     & & 199 & 190=2\cdot 5\cdot 19  \\ \hline
  4921=7\cdot 19\cdot 37 & 4950=2\cdot 3^{2}\cdot 5^{2}\cdot 11
     & 7 & 10=2\cdot 5  \\
     & & 19 & 22=2\cdot 11  \\
     & & 37 & 45=3^{2}\cdot 5  \\ \hline
\end{array}
\]
\caption{Elliptic Korselt numbers for $E:y^2+xy+3y=x^3+2x^2+4x$}
\label{table:korselt12340}
\end{table}

\subsection*{Acknowledgements}
I would like to thank Felipe Voloch and \'Alvaro Lo\-za\-no-Robledo for
the observation in Remark~\ref{remark:canlift}.  The
research described in in this note was partly supported by the NSF
(grant no. DMS-0854755).





\end{document}